\documentclass{amsart}

\usepackage{amssymb}
\usepackage[shortlabels]{enumitem}
\usepackage{tikz}
\usepackage{subfigure, float}
\usepackage{comment}
\usepackage{hyperref}
\hypersetup{colorlinks,linkcolor={red},citecolor={olive},urlcolor={red}}
\usepackage{mathtools}
\mathtoolsset{showonlyrefs}

\newcommand{\f}{\frac}
\newcommand{\ind}[1]{\mathbf{1}{\{ #1 \}}}
\newcommand{\T}{\mathcal{T}}
\renewcommand{\root}{\rho}
\newcommand{\icx}{\preceq_{\ICX}}
\renewcommand{\phi}{\varphi}

\DeclareMathOperator{\ICX}{icx}

\DeclareMathOperator{\var}{var}

\usepackage{theoremref}

\newtheorem{theorem}{Theorem}
\newtheorem{lemma}{Lemma}[section]
\newtheorem{proposition}[theorem]{Proposition}
\newtheorem{corollary}[theorem]{Corollary}

\newtheorem*{GP}{Theorem 3.4 \cite{tree}}

\theoremstyle{remark}
\newtheorem{remark}[lemma]{Remark}
\theoremstyle{definition}

\title{Parking on supercritical Galton-Watson trees}

	\author[R.~Bahl]{Riti Bahl}
	\email{\texttt{rb9956@bard.edu}}

	\author[P.~Barnet]{Philip Barnet}
	\email{\texttt{pb7734@bard.edu}}

	\author[M.~Junge]{Matthew Junge}
	\email{\texttt{mjunge@bard.edu}}

\thanks{All three authors were partially supported by NSF DMS Grant \#1855516.}

\begin{document}
\maketitle

\begin{abstract}
At each site of a supercritical Galton-Watson tree place a parking spot which can accommodate one car. Initially, an independent and identically distributed number of cars arrive at each vertex. Cars proceed towards the root in discrete time and park in the first available spot they come to. Let $X$ be the total number of cars that arrive to the root.
%
%
Goldschmidt and Przykucki proved that  $X$ undergoes a phase transition from being finite to infinite almost surely as the mean number of cars arriving to each vertex increases. 
%
We show that $EX$ is finite at the critical threshold, describe its growth rate above criticality, and prove that it increases as the initial car arrival distribution becomes less concentrated. For the canonical case that either 0 or 2 cars arrive at each vertex of a $d$-ary tree, we give improved bounds on the critical threshold and show that $P(X=0)$ is discontinuous as a function of $\alpha$ at $\alpha_c$.
\end{abstract}

\section{Introduction}

Parking, introduced over fifty years ago \cite{KW}, is a stochastic process at the intersection of probability and combinatorics. 
The \emph{parking process} on a tree $\T$ with root $\rho$ begins with a parking spot at each vertex. Initially, $\eta_v$ cars  arrive to each vertex $v \in \T$ and move  towards the root in discrete time steps. When a car arrives at an available spot, the car parks there and the spot becomes unavailable. If multiple cars arrive to the same available spot, then one is chosen uniformly at random to park there. The remaining cars continue moving towards the root. Let $X$ be the total number of cars that arrive to $\rho$. 

There has been significant progress on understanding how $X$ behaves when $\T$ is a critical Galton-Watson tree \cite{tree, chen2019parking, curien2019phase}. Less is known about the case that $\T$ is supercritical. The point of this article is to make some progress on this case and develop some machinery that might aid future work. In particular, we prove that $EX$ is finite at criticality, describe the expected growth rate of the number of arrivals,  observe that $X$ increases as the $\eta_v$ become less concentrated, and provide some concrete bounds concerning a simple case. 

Suppose that the offspring distribution of the Galton-Watson tree $\mathcal T$ is described by the nonnegative integer-valued random variable $Z$ with $E Z = \lambda >1$. Additionally, assume that the $\eta_v$ are independent and identically distributed (i.i.d.)\ as $\eta(\alpha)$, which is a family of random variables that is stochastically increasing in $\alpha = E \eta(\alpha)$. Stochastically increasing means that $P(\eta(\alpha) \geq x) \leq P(\eta(\alpha') \geq x)$ for all $x \geq 0$ and $\alpha\leq \alpha'$.
For this setting, Goldschmidt and Przykucki proved the following.

\begin{GP} There exists $\alpha_c \in (0,1)$ such that if $\alpha  < \alpha_c$, then 
\begin{align}E X = \f{\lambda - \alpha - \lambda P(X=0)}{\lambda -1},
	\end{align}
 while if $\alpha > \alpha_c$, then, conditionally on the non-extinction of the tree, $X = \infty$ almost surely.
\end{GP}

Unless stated otherwise, we let $\eta(\alpha)$ and $\alpha_c$ be as in \cite[Theorem 3.4]{tree}.  What happens when $\alpha = \alpha_c$ was left open. Our first result shows that $EX$ is finite at criticality. 
 
\begin{theorem}\thlabel{thm:main}
	For all $\alpha \leq \alpha_c$ it holds that 
	\begin{align}E X = \f{\lambda - \alpha -  \lambda P(X=0)}{\lambda -1}.\label{eq:EX}\end{align} 
\end{theorem}

Determining the value of $\alpha_c$ remains an open problem. Let $X_n$ be the number of cars that arrive to $\rho$ up to time $n$. One consequence of the proof of \thref{thm:main} is a formula for the growth rate of $E X_n$ as well as a characterization for $\alpha_c$ in terms of the first time that a car arrives to the root.

\begin{proposition} \thlabel{prop:alpha_c}
If $\tau$ is the time that the first car arrives to $\rho$, then 
\begin{align}
\lim_{n \to \infty} \f {E X_n }{\lambda^n} = \f{\lambda}{\lambda-1} (\alpha - E \lambda^{-\tau}).\label{eq:growth}
\end{align}
Moreover,
\begin{align}
\alpha_c = \sup \{ \alpha \colon E \lambda^{-\tau}= \alpha \}.\label{eq:alpha_c}
\end{align}
\end{proposition}

One particularly simple choice for the distribution of $\eta(\alpha)$ is that it takes value $2$ with probability $\alpha/2$ and otherwise is $0$. Since at time $0$ the number of cars at each site is a Bernoulli random variable, we will refer to the parking process with this distribution as \emph{Bernoulli parking}. 
As a further simplification, we consider Bernoulli parking on $\mathbb T_d$ the infinite $d$-ary tree in which each vertex has $d$ children. This case is presented as canonical in \cite{tree, collet1983study, hu2018free}. We denote the critical threshold for this specific setting by \begin{align}
\alpha_c(d) = \text{critical value in Bernoulli parking on $\mathbb T_d$}.
\label{eq:alpha_c}
\end{align}
\cite[Theorem 3.4]{tree} tells us that $P(X=0) = 0$ for $\alpha > \alpha_c$. Extending \thref{thm:main}, we show that the probability that no cars arrive to $0$ is discontinuous at $\alpha = \alpha_c(d)$. 
\begin{proposition}\thlabel{thm:discontinuous}
For $\alpha \leq \alpha_c(d)$ in Bernoulli parking on $\mathbb T_d$ we have $P(X=0) > 0$.	
\end{proposition}
\thref{thm:main} and \thref{thm:discontinuous} are different than what occurs when $\T$ is a critical Galton-Watson tree conditioned to be infinite. In this setting $E X = \infty$ and $P(X=0)=0$ at criticality. See the discussion of results from \cite{tree,chen2019parking, curien2019phase} following  the statements of our results. 

It was shown in \cite[Theorem 3.5]{tree} that $0.03125 \leq \alpha_c(2) \leq	 0.50$. We can use ideas from the proof of \thref{thm:discontinuous} to give a greatly improved upper bound, and a small tweak to the proof of \cite[Theorem 3.5]{tree} to slightly improve the lower bound.

\begin{proposition}\thlabel{thm:bounds}
$0.03175 <\alpha_c(2) <.08698$
\end{proposition}

 The calculation for the upper bound is computer-assisted which in theory gives arbitrarily close upper bounds. Runtime with exact precision quickly becomes an issue. Truncating the decimals in our calculations allows us to compute further and still have a rigorous bound, but at the cost of some accuracy. Nonetheless, we believe that the upper bound is very close to the correct value of $\alpha_c(2)$. Allowing for rounding error, the evidence suggests that $\alpha_c(2) \approx .0863$. See the proof of \thref{thm:bounds} for more details.
It is natural to ask how $\alpha_c(d)$ changes as $d$ is increased. A straightforward generalization of \cite[Theorem 3.5]{tree}, gives that $\alpha_c(d) \approx d^{-2}$. 

\begin{proposition} \thlabel{thm:pcd}
It holds for all $d \geq 2$ that 
\begin{align}\f{1}{2 e^2} d^{-2}\leq  \alpha_c(d) \leq 2d^{-2}.\label{eq:pcd}	
\end{align}

\end{proposition}
Comparing the following result to \thref{thm:pcd}, we see that the location of the phase transition on $\mathbb T_d$ depends on more than just the mean of the car arrival distribution.

\begin{proposition}\thlabel{prop:different}
Let $\alpha'_c(d)$ be the critical threshold for the parking process on $\mathbb T_d$ with $\eta(\alpha) = 3$ with probability $\alpha /3$ and $0$ otherwise. It holds that 
$$\alpha_c'(d) \leq 3 d^{-3}.$$ Combining this with \eqref{eq:pcd} gives $\alpha'_c(d) < \alpha_c(d)$ for large enough $d$.
\end{proposition}

That $\alpha_c$ depends on the distribution of the $\eta(\alpha)$ is part of a more general property of the parking process. Namely, that $X$ increases as $\eta(\alpha)$ becomes less concentrated. On critical Galton-Watson trees with $\alpha$ fixed, Curien and  H\'enard proved in \cite{curien2019phase} that $\alpha_c$ decreases linearly in $\var(\eta)$ (see \eqref{eq:general_PT}.) We prove a more general, albeit less precise, result. 

Given random variables $X$ and $Y$ taking values in $[0,\infty)$, we say that $Y$ dominates $X$ in the \emph{increasing convex} order if for all bounded, increasing convex functions $\phi\colon [0,\infty) \to \mathbb R$ it holds that $E \phi(X) \leq E \phi(Y)$. Denote this ordering by $X \icx Y$. Roughly speaking, the less concentrated a distribution is, the larger it is in the increasing convex order. As the identity function is convex, when $X \icx Y$ we have $EX \leq EY$. Moreover, if $E X = E Y$, then, since $x^2$ is increasing and convex, we have $X \icx Y$ implies that $\var(X) \leq \var(Y)$. See \cite{SS} for a thorough survey of stochastic orders. We show for all Galton-Watson trees (not just supercritical) that $X$ increases in the increasing convex order when $\eta$ does. Consequently, so does $EX$. 

\begin{theorem}\thlabel{thm:icx}
Let $X$ and $X'$ denote the total number of cars that arrive to $\rho$ for the parking process on a Galton-Watson tree with car arrival distributions $\eta$ and $\eta'$, respectively. If  $\eta \icx \eta'$, then $X \icx X'.$ 
\end{theorem}

An equivalent stochastic order is considered in \cite{JJ} for an interacting particle system known as the frog model. Unlike parking, the number of visits to the root in this process decreases if the initial particle distribution is replaced by one with the same mean that is larger in the increasing convex order. An analogous effect occurs for the limiting shape in first passage percolation \cite{BK, Marchand}. 

We expect that a similar statement as \thref{thm:icx} holds for parking on arbitrary trees. The proof we envision would be technical and we felt would distract from the main goals of this paper. We plan to tackle this in a followup work. For now we provide a corollary, which says that Bernoulli parking gives the maximal critical threshold among all arrival distributions whose supports do not include $\{1\}$. So, our estimates in \thref{thm:bounds} and \thref{thm:pcd} hold for a large family of arrival distributions.  

\begin{corollary} \thlabel{cor:max}
	If $\alpha_c$ is the critical value for Bernoulli parking on a Galton-Watson tree $\T$, then $\alpha_c' \leq \alpha_c$ with $\alpha_c'$ the critical value for parking on $\T$ with any other family $\eta'(\alpha)$ of car arrival distributions satisfying the hypotheses of \cite[Theorem 3.4]{tree} and whose support does not include $\{1\}$.
\end{corollary}

\subsection{Discussion}

Parking dynamics were introduced by Konheim and Weiss for $\T = [1,n]$ the path on $n$ vertices and $\rho =1$ \cite{KW}. They fixed a parameter $\alpha \in (0,1]$ and placed $\lceil \alpha n \rceil$ cars uniformly at random on $[1,n]$. Let $A_n$ be the event that every car parks. Their main result was an asymptotic formula for the probability a given configuration is a parking function
\begin{align}\lim_{n \to \infty} P(A_n) = (1-\alpha) e^{\alpha}.\label{eq:KW}
\end{align}
There has since been significant followup study of the combinatorial structures that arise from parking functions. See the work of Stanley  and Pitman \cite{stanley1, stanley2, stanley3} as well as Diaconis and Hicks \cite{persi_parking}.

Notice that \eqref{eq:KW} is never equal to zero. Lackner and Panholzer showed that there is a phase transition when $\T$ is a uniformly random tree on $n$ vertices and $\lceil \alpha n \rceil$ cars are placed uniformly at random throughout the vertices. Again letting $A_n$ be the event that every car parks, they proved that $P(A_n)$ has limiting behavior 
\begin{align}\lim_{n \to \infty} P(A_n) = \begin{cases}  
		\f{\sqrt{1- 2 \alpha}}{1-\alpha}, & 0\leq \alpha \leq  1/2\\
		0, & \alpha \geq 1/2  
	\end{cases} \label{eq:LP}.
\end{align}

Goldschmidt and Przykucki studied the natural limiting case of \cite{uniform_parking} in \cite{tree}. They let $\T$ be a  Galton-Watson tree with a Poisson with mean $1$ offspring distribution conditioned to be infinite. Each $\eta_v$ is an independent Poisson random variable with mean $\alpha$. For $A$ the event that every car parks, they showed that $P(A)$ has the same formula as \eqref{eq:LP}. Furthermore, they deduced the main theorem of \cite{uniform_parking} as a corollary of their theorem on the infinite tree. 
 Recently, Chen and Goldschmidt proved a similar result for when $\T$ is the limiting tree from a sequence of uniformly random rooted plane trees. In this case, the phase transition occurs at $\alpha = \sqrt 2 -1 \approx 0.4142$, rather than $1/2$.
 
The parking process has been studied from two alternative perspectives. Jones viewed parking as a model for runoff of rainfall in \cite{jones2019runoff}. Parking can also be thought of as an interacting particle system with mobile particle (cars) and stationary particles (spots) which mutually annihilate upon colliding. This was first studied  on the integer lattice with cars performing simple random walk by Cabezas, Rolla, and Sidoravicius under the name \emph{particle-hole model} \cite{Cabezas2014}. Later Damron, Gravner, Junge, Lyu, and Sivakoff studied these dynamics on transitive unimodular graphs \cite{parking}. This is a special case of two-type diffusion-limited annihilating systems studied in the physics literature \cite{chem1, AB1} and also by mathematicians \cite{BL4}. More recently, Przykucki, Roberts, and Scott studied the parking process with cars performing simple random walk on the integers \cite{parking_on_integers}.

Returning our discussion to parking on  trees, both \cite{tree} and \cite{chen2019parking} rely on explicit formulas for the generating function of $X$ when $\T$ is an unconditioned critical Galton-Watson tree. This is made possible through a self-similarity present in the parking process on Galton-Watson trees and then additional nice properties from the underlying offspring distributions (Poisson and Geometric with mean 1 in \cite{tree} and \cite{chen2019parking}, respectively). Namely, if $Z$ is the number of children of $\rho$, then 
\begin{align}
X = \eta_\rho  + \sum_{i=1}^Z ( X^{(i)} - 1)^+,\label{eq:RDE1}
\end{align}
where the $X^{(i)}$ are i.i.d.\ copies of $X$ and $x^+ = \max( 0,x).$ Similar equations as \eqref{eq:RDE1}:
$$Y_{n+1} = (Y_n^{(1)} + Y_n^{(2)} + \cdots +Y_n^{(Z)} -1)^+,$$
 which are related to spin-glasses, are referred to as Derrida-Retaux models see \cite{collet1983study, hu2018free, Hu_2019}. The critical value for such equations and the behavior at criticality is much better understood than \eqref{eq:RDE1}.

Very recently, in \cite{curien2019phase}, Curien and H\' enard confirmed a conjecture from \cite{tree} by generalizing the phase transition results from \cite{tree} and \cite{chen2019parking} to arbitrary Galton-Watson trees whose offspring distributions have mean 1 and finite variance $\Sigma^2$. They proved that when the $\eta_v$ are i.i.d.\ with mean $\alpha$ and variance $\sigma^2$ and $\T$ is such a Galton-Watson tree conditioned to be infinite, a phase transition for $E X$ occurs when  
\begin{align}\theta:= (1-\alpha)^2 - \Sigma^2( \sigma^2 + \alpha^2 -\alpha) =0\label{eq:general_PT}.
\end{align}
For example, if the offspring distribution is Poisson with mean 1 and arrival distribution is Poisson with mean $\alpha$, solving \eqref{eq:general_PT} gives $\alpha =1/2$ as in \cite{tree}.

What happens at criticality? For the setting in \cite{tree}, Goldschmidt and Przykucki proved that  $E X$ undergoes a discontinuous phase transition on the critical Poisson Galton-Watson tree:
\begin{align}
	EX = \begin{cases} 1- \sqrt{ 1- 2\alpha} , &\alpha \leq 1/2	 \\ \infty, &\alpha >1/2
 \end{cases}.	
\end{align}
 In particular, $E X = 1$ when $\alpha = 1/2$. Similar behavior was observed for $EX$ in the setting in \cite{chen2019parking}. All of this is covered by the main theorem of \cite{curien2019phase} which implies, among other things, that 
\begin{align}
E (X-1)^+ = \begin{cases} \f{1- \sqrt{ \theta} + \alpha}{\Sigma^2}, & \theta \geq 0  \\  	
	\infty, & \theta < 0
 \end{cases} 
\end{align}
with $\theta$ defined at \eqref{eq:general_PT}.

	\begin{figure}
		\includegraphics[width = .6 \textwidth]{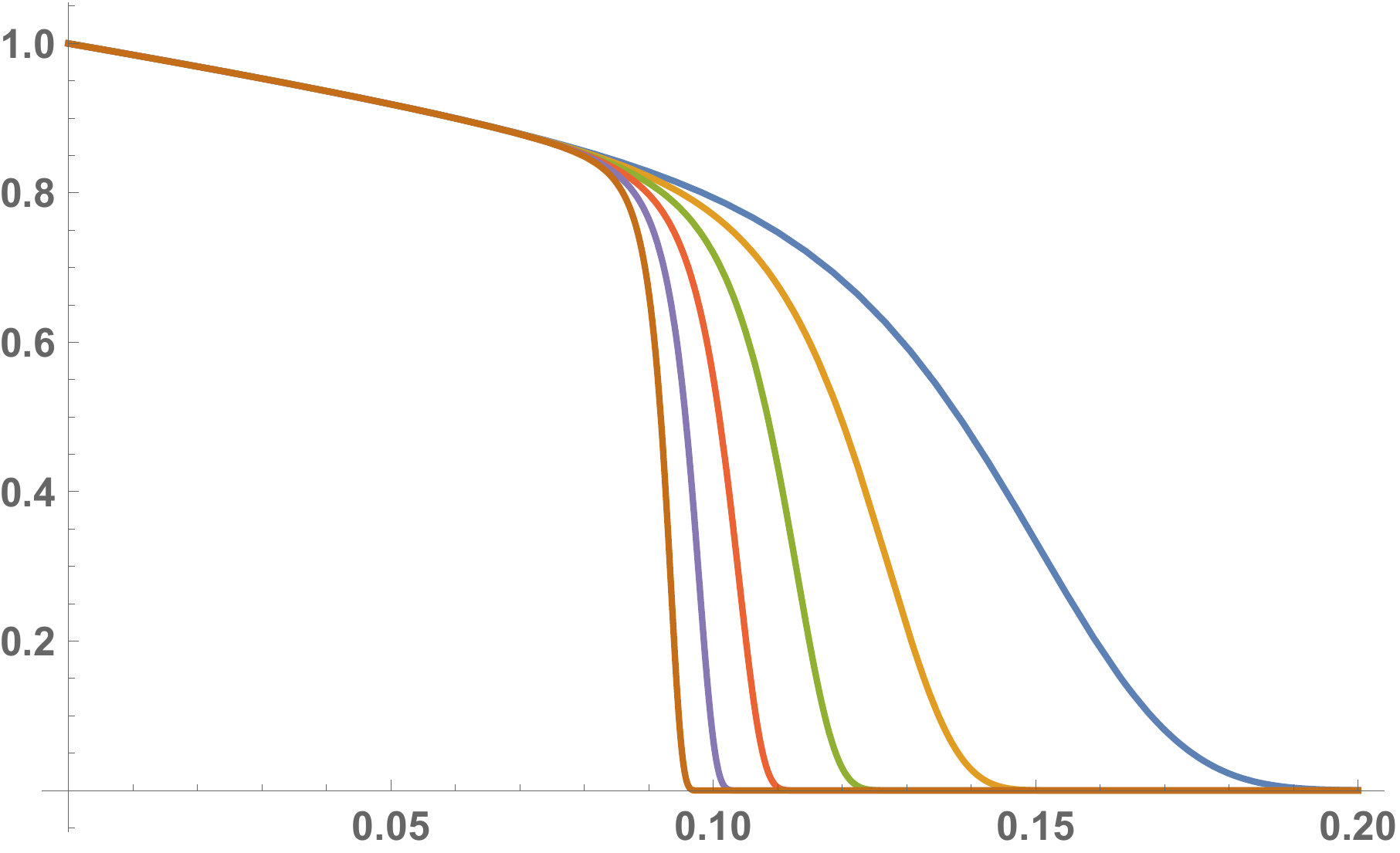}	
		\caption{Plots of $P(X_{n} = 0)$ in Bernoulli parking on $\mathbb T_2$ for $n=10,15,20,30,35,40$ (arranged right to left) and  $\alpha \in [0,0.2].$ The plotted values possibly have small floating point inaccuracies for large $n$. We do not have an explicit formula for $P(X=0)$, but these curves are increasingly accurate approximations. The fact that $P(X=0)$ is discontinuous at $\alpha_c(2)$ (\thref{thm:discontinuous}) becomes increasingly apparent.} \label{fig:q}
	\end{figure}

As remarked earlier, less is known about the phase transition on supercritical trees. It would be good to have an exact formula for $\alpha_c(d)$, and more generally for the critical threshold on supercritical Galton-Watson trees analogous to the main theorem of \cite{curien2019phase}. Even a heuristic for where the threshold should be would be nice to have. It is unclear to us if, like for critical Galton-Watson trees, the threshold only depends on the mean and variance of $\eta$. One clear difference from the formula for $\alpha_c$ at \eqref{eq:general_PT}, is that even if the variance of $Z$ is $0$, i.e. $Z \equiv d$, then the critical threshold can change. So it seems likely that, if a closed form for $\alpha_c$ exists on supercritical Galton-Watson trees, it is more involved than \eqref{eq:general_PT}. 

Although we prove that $P(X=0) >0$ for $\alpha \leq \alpha_c(d)$, we do not have a closed formula for $P(X=0)$, nor for $EX$. Nor do we have a conjecture, but in Figure \ref{fig:q} we give a few plots of $P(X_{n} =0)$ for Bernoulli parking on $\mathbb T_2$, where $X_n$ is the number of cars that arrive to $\rho$ up to time $n$. 

\subsection{Organization}
In Section \ref{sec:critical} we prove \thref{thm:main}, \thref{prop:alpha_c}, and \thref{thm:discontinuous}. Section \ref{sec:bernoulli} has the results for Bernoulli parking: \thref{thm:bounds}, \thref{thm:pcd}, and \thref{prop:different}. Section \ref{sec:icx} has the stochastic comparison results: \thref{thm:icx} and \thref{cor:max}.

\section{Critical behavior} \label{sec:critical}
Recall that $X$ is the total number of cars that arrive at $\rho$ when $\mathcal T$ is a Galton-Watson tree with offspring distribution $Z$ satisfying $EZ = \lambda >1$. The number of cars arriving to the site $v$ is $\eta_v$ which has distribution $\eta$. Let $X_n$ be the number of cars that arrive to $\rho$ up to time $n$. We let $X_0 = \eta_\rho$. Let $q_n = P(X_n=0)$.
 
Our starting point is a closed formula for $E X_{n+1}$. Define the functions
	\begin{align}G_n(\alpha) = \sum_{i=0}^n \lambda^{-i} q_i; \quad F(\alpha) &= \f{\lambda(1-\alpha)}{\lambda -1}; \quad 
				C(\alpha) = \f{1-\alpha}{\lambda-1}. \label{eq:functions}
	\end{align}	
Also, let $G(\alpha) = \lim_{n \to \infty} G_n(\alpha)$. 

\begin{proposition}\thlabel{prop:Xn} Let $G_n, F$ and $C$ be as in \eqref{eq:functions}. It holds for all $n \geq 0$ that 
	\begin{align}
EX_{n+1}
	&= \left( G_n(\alpha) - F(\alpha) \right) \lambda^{n+1} + C(\alpha). \label{eq:X_rec}
\end{align}
\end{proposition}

\begin{proof}
The truncated analogue of \eqref{eq:RDE1} is
\begin{align}
X_{n+1} = \eta_\rho + \sum_{i=1}^Z (X_n - 1)^+,\label{eq:RDEn}
\end{align}
which follows from self-similarity of $\mathcal T$.
Taking the expected value of both sides gives
\begin{align}
E X_{n+1} &= \alpha + \lambda (E X_n - P(X_n>0))
\end{align}
Iterating the recursion yields
\begin{align}
E X_{n+1} &= \lambda^{n+1} E X_0 + \sum_{i=0}^n \lambda^i \alpha - \sum_{i=1}^{n+1} \lambda^i (1-q_{n-i+1}) ),
\end{align}
which simplifies to \eqref{eq:X_rec} after expanding the $\sum \lambda^{-i}$ terms and factoring out $\lambda^{n+1}$.
\end{proof}

\thref{prop:Xn} gives us a necessary and sufficient condition to have $E X < \infty$.

\begin{lemma}\thlabel{lem:iff}
	$G(\alpha) - F(\alpha) = 0$ if and only if $E X < \infty$.
\end{lemma}
\begin{proof}
First note that $F(\alpha), C(\alpha) >0$ for $\alpha <1$. For such $\alpha$, we must have $G(\alpha) - F(\alpha) \geq 0$. Otherwise, since $G_n \uparrow G$, we would have $G_n(\alpha) - F(\alpha) <-\delta$ for some $\delta >0$, which gives the contradiction that $E X_n \downarrow - \infty$. If $G(\alpha) - F(\alpha) =0$, then \eqref{eq:X_rec} implies that $EX \leq C(\alpha)$ for all $ n \geq 1$. The monotone convergence theorem implies that $E X_n \uparrow EX \leq  C(\alpha) < \infty $. 

If $F(\alpha) - G(\alpha) >0$ then, since $G_n\uparrow G$ is strictly increasing in $n$, we have $G_N(\alpha) - F(\alpha) = \delta$ for some $\delta >0$ and large enough $N$. The formula at \eqref{eq:X_rec} implies that $E X_n \geq \delta d^n$ for all $n \geq N$, and thus $E X = \infty$.  
\end{proof}

To describe what happens at $\alpha = \alpha_c$ we require continuity of $G$, which relies on continuity of $q_n$ in $\alpha$.  First we prove that the distribution of $\eta(\alpha)$ is continuous in $\alpha$.

\begin{lemma} \thlabel{lem:continuous}
Suppose that $(\eta(\alpha))$ is a stochastically increasing family of random variables supported on the nonnegative integers with $E \eta(\alpha) = \alpha$. It holds for all $k \geq 0$ that $P(\eta(\alpha) = k)$ is a continuous function in $\alpha$.
\end{lemma}

\begin{proof}
Let $\alpha' < \alpha$. Since $\alpha = E \eta(\alpha) = \sum_{m \geq 0} P(\eta(\alpha) > m)$, we  write
$$\alpha - \alpha ' = E \eta(\alpha) - E \eta(\alpha') = \sum_{m=0}^\infty [ P(\eta(\alpha) > m) - P(\eta(\alpha') > m)].$$
Because $(\eta(\alpha))$ is stochastically increasing, each summand is positive. Thus, $P(\eta(\alpha) > m) - P(\eta(\alpha') >m) \leq \alpha - \alpha'$, which can be made arbitrarily small. It follows that $P(\eta(\alpha) > m)$ is continuous in $\alpha$ for all $m \geq 0$. This implies that $1- P(\eta(\alpha) >m) = \sum_{k=0}^m P(\eta(\alpha) = k)$ is also continuous. Iteratively applying this fact for $m=0,1,\hdots$ gives that $P(\eta(\alpha) = k)$ is continuous for all $k \geq 0$. 
\end{proof}

\begin{lemma} \thlabel{lem:q_cont}
$q_n$ is continuous in $\alpha$ for all $n\geq 0$. 	
\end{lemma}

\begin{proof}
Let $\mathcal T_n$ denote the subset of $\mathcal T$ containing all vertices within distance $n$ of $\rho$. Fix $N >0$ and partition
\begin{align}q_n &= P(X_n =0, |\T_n| \leq N) + P(X_n =0, |\T_n| > N)	\\
				&\leq P(X_n =0, |\T_n| \leq N) + P(|\T_n| > N).	\label{eq:q_breakdown}
\end{align}
Using Markov's inequality we have $P(|\T_n| > N ) \leq \lambda^n / N$, and can be made arbitrarily small for fixed $n$. 

Observe that there are finitely many trees $\T_n \leq N$, and the event $X_n=0$ requires that all $\eta_v \leq N$ with $v \in \T_n$, otherwise $\rho$ would be visited. Thus, $P(X_n =0, |\T_n\| \leq N)$
is a finite sum involving only products of the probabilities $P(\eta = k)$ for $k \leq N$. By \thref{lem:continuous}, this is continuous. 
 Hence, for $\alpha \in (0,1)$ and any  $\epsilon >0$, we choose $N$ so that $P(|\T_n| >N) < \epsilon /3$ and $\delta$ so that 
$$|P( X_n(\alpha) = 0, |T_n| \leq N) - P(X_n(\alpha') = 0), |T_n| \leq N) |<\epsilon /3$$
for all $|\alpha - \alpha'| < \delta$. Here 
$X_n(\alpha)$ signifies the dependence of $X_n$ on $\alpha$. Applying this to \eqref{eq:q_breakdown} gives for all $|\alpha - \alpha'|<\delta$ we have $|q_n(\alpha) - q_n(\alpha')| < \epsilon,$ and so $q_n$ is continuous at $\alpha$. 
\end{proof}

Now we can prove that $E X$ is finite when $\alpha = \alpha_c$.

\begin{proof}[Proof of \thref{thm:main}]
By \thref{lem:iff}, it suffices to prove that $G(\alpha_c) - F(\alpha_c)	=0$. 
We claim that $G(\alpha)$ is continuous for all $\alpha \in (0,1)$. 
By our hypothesis that $P(\eta = k)$ is continuous in $\alpha$ and \thref{lem:q_cont}, the $q_n$ are continuous functions of $\alpha$. It follows that each $G_n = \sum_{i=0}^n\lambda^{-i} q_i$ is continuous. Moreover, the convergence $G_n\uparrow G$ is uniform since
$$G(\alpha) - G_n(\alpha) = \sum_{i >n} \lambda^{-i} q_i \leq \sum_{i>n} \lambda^{-i}$$
which can be made arbitrarily small for all sufficiently large $n$. A uniformly convergent sequence of continuous functions is continuous, so $G$ is continuous.  As $F$ is also continuous, it follows that $G(\alpha) - F(\alpha)$ is continuous for all $\alpha \in (0,1)$. 

	\cite[Theorem 3.4]{tree} tells us that $\alpha<\alpha_c$ implies $EX < \infty$. Thus, \thref{lem:iff} gives that $G(\alpha) -F(\alpha) =0$ for all $\alpha < \alpha_c$. Continuity of $G-F$ implies that $$G(\alpha_c)-F(\alpha_c)  = \lim_{\alpha \to \alpha_c^-} G(\alpha) - F(\alpha)= 0.$$
	Thus, $EX < \infty$ when $\alpha = \alpha_c$. The explicit formula \eqref{eq:EX} follows from taking expectation in \eqref{eq:RDE1} and solving for $EX$, which is valid whenever $EX <\infty$. 
\end{proof}
	
\begin{proof}[Proof of \thref{prop:alpha_c}]
It follows from \thref{lem:iff} that
	$$\alpha_c = \sup\{ \alpha \colon G(\alpha) - F(\alpha) = 0\}.$$
	Notice that $q_n = P(X_n = 0) = P(\tau > n)$. We then have
	\begin{align}
		G(\alpha) = \sum_{i=0}^\infty \lambda^{-i} P(\tau >i)	 = \sum_{i=0}^\infty \sum_{m>i} \lambda^{-i} [ P(\tau =m)	+ P(\tau = \infty)].
	\end{align}
Apply Fubini's theorem to the first term and separate out the $\sum_{i=0}^\infty \lambda^{-i} P(\tau = \infty)$ part to write this as
\begin{align}
G(\alpha) &= \sum_{m>0} \sum_{i=0}^{m-1} \lambda^{-i} P(\tau = m) + \sum_{i=0}^\infty \lambda^{-i} P(\tau = \infty)\\
&= \sum_{m >0} \f{ \lambda - \lambda^{-m+1} }{\lambda -1} P(\tau = m) + \f \lambda {\lambda - 1} P(\tau = \infty)\\
&= \f \lambda {\lambda -1} \left( P(\tau= \infty)+ \sum_{m>0} [P(\tau =m) - \lambda^{-m} P(\tau = m) ]   \right).
\end{align}
After grouping the $P(\tau = m)$ terms and the $\lambda^{-m} P(\tau =m)$ terms and accounting for the fact that both are missing $P(\tau =0)$, this simplifies to 
\begin{align}
G(\alpha) &=\f \lambda {\lambda -1} \left( 1 - E \lambda^{-\tau} \right). 
\end{align}
Now, subtracting $F(\alpha) = \lambda (\lambda -1)^{-1}(1-\alpha)$ and simplifying a bit gives
\begin{align}
G(\alpha) - F(\alpha) = \f \lambda {\lambda -1} \left(\alpha - E \lambda^{-\tau} \right).\label{eq:form}
\end{align}

\thref{prop:Xn} tells us that 
$$\f{E X_{n+1}}{\lambda^n} = (G_n(\alpha) - F(\alpha) ) + C(\alpha) \lambda^{-n}.$$
Taking the limit of the above and applying the equality at \eqref{eq:form} gives  \eqref{eq:growth}. As for \eqref{eq:alpha_c}, by \thref{lem:iff} we have $EX$ is finite if and only if $\alpha - E \lambda^{-\tau}=0$, thus $\alpha_c$ is the largest solution to this equation.

\end{proof}

\section{Bernoulli parking on $\mathbb T_d$} \label{sec:bernoulli}

We start with a quick proof that $P(X=0)$ is discontinuous. 

\begin{proof}[Proof of \thref{thm:discontinuous}]
Suppose that $P(X=0) >0$ for some $\alpha \leq \alpha_c$. Using \eqref{eq:RDE1} we have
	$P(X=k)$ with $k \leq j+1$ can be written as
	$$P(X=j) = P \left( \eta_\root + \sum_{i=1}^d (X^{(i)} -1)^+  = k \right).$$
	Thus, we can write $P(X=j)$ as a convolution involving $P(X=k)$ for $0 \leq k \leq j+1$. 
	 For example,
	$$P(X=0) = (1-(\alpha/2))^{d+1}(P(X=0) + P(X=1))^d.$$
	Given that $P(X=0) = 0$, we can iteratively deduce that $P(X=j) = 0$ for all $j \geq 0$. This implies that $P(X=\infty) = 1$, which contradicts that $EX < \infty$ from \thref{thm:main}. Thus, $P(X=0) >0$ when $\alpha \leq \alpha_c(d)$. 
\end{proof}

\subsection{Bounds for $\alpha_c(2)$}

We now turn our attention to improving the estimates on $\alpha_c(2)$. The idea is  generate closed forms for $q_n$ using a recursive relationship.  Let $p = \alpha /2$. For $n \geq 1$ define $V_{n}= (X_{n}-\eta_\rho)^+$ to be the number of cars that arrive to $\rho$ between time $1$ and $n$. Let $r_{n,j} = P(V_{n} = j)$. Notice we have the simple relationship
\begin{align}q_{n} = (1-p)r_{n,0}.\label{eq:qV}\end{align}
 The following lemma describes a recursion satisfied by the $r_{n,j}$.

\begin{lemma}\thlabel{lem:q_recursion}
Let $n \geq 1$. Set $r_{n,j}=0$ for $j<0$ and $j> 2^{n+1}-2$. For $n=0$,
	\begin{align}
		r_{1,0} = (1-p)^2, \quad 
		r_{1,1} = 2p(1-p), \quad 
		r_{1,2} = p^2.
	\end{align} When $j=0$ we have
	\begin{align}
		r_{n+1,0} &= (1-p)^2(r_{n,0} + r_{n,1})^2.
	\end{align}
 It holds for all $0 < j \leq 2^{n+1} -2$ that
	\begin{align}	
		 r_{n+1,j} &= p^2 \left(\sum _{k=0}^{j-2} r_{n,k} r_{n,j-k-2} \right) \\ 
				& \quad + 2 p (1-p) \left(r_{n,0} r_{n,j-1} + \sum _{k=0}^{j-1} r_{n,k} r_{n,j-n}\right)	 \\
				& \qquad + (1-p)^2\left(2r _{n,0} r_{n,j+1}+ \sum_{k=1}^{j+1} r_{n,k} r_{n,j-k+2}\right).
	\end{align}
	\end{lemma}
\begin{proof}
This follows from \eqref{eq:RDE1}. Label the two children of the root as $x$ and $y$. The formulas for $r_{1,j}$ come from the fact that $V_1= \ind{\eta_x = 2} + \ind{\eta_y=2}$ is a Binomial random variable.  The formula for $r_{n+1,0}$ comes from the requirement that $\eta_x, \eta_y =0$ and that no more than one car visits each of $x$ and $y$, respectively. Clearly, $r_{n,j} =0$ for $j<0$, and since the number of vertices in $\mathbb T_2$ up to distance $n$ from $\rho$ is $2^{n+1}-1$. The formula for $r_{n+1,j}$ comes from conditioning on $\eta_x, \eta_y$ and then partitioning on $k$ cars arriving at $x$. Special considerations need to be made when $k = j+1$ since, in this case, either $0$ or $1$ cars can arrive to $y$. Similarly for when $k=0$. 
\end{proof}

\begin{proof}[Proof of \thref{thm:bounds}]
We start with the upper bound. It follows from \thref{lem:iff} that $G(\alpha) - F(\alpha) >0$ if and only if $\alpha> \alpha_c$. Since $G_n \uparrow G$, if 
\begin{align} G_n(\alpha) - F(\alpha) >0\label{eq:sufficient}	
\end{align}
then so is $G(\alpha) - F(\alpha)$. Thus, if we can find a pair $n,\alpha$ satisfying \eqref{eq:sufficient}, then $\alpha$ is an upper bound on $\alpha_c$. 

Using \thref{lem:q_recursion} and \eqref{eq:qV}, it is not so taxing to write out $q_n$ for small values of $n$ by hand, but this quickly becomes intractable. A computer can calculate  $q_n$ for much larger values of $n$, but still is limited, since the number of terms and degree of the polynomials grow exponentially. Rounding error makes any estimates obtained with floating point calculations non-rigorous. 

We avoid this issue by truncating all numbers (at the 200th decimal place). Let $g_n$ be the analogue of $G_n$, but with all decimals truncated. Since every summand in the formula for $q_n$ from \thref{lem:q_recursion} and \eqref{eq:qV} are positive, truncating gives a lower bound:  $g_n(\alpha) \leq G_n(\alpha)$ for all $\alpha$. 
It only takes a few seconds to show that for $\alpha_0 = 0.08698$ we have 
$$g_{50}(\alpha_0) - F(\alpha_0) >0,$$ and thus $G(\alpha_0) - F(\alpha_0)$ is also positive. 
Thus, $\alpha_c<0.08698$. 

There is some loss of accuracy from truncation. However, being able to compute further is better than computing with perfect accuracy. For example, it took several hours to show that $G_{23}(112/1000) - F(112/100) >0$ working with rational numbers. For $n=23$ most of the fractions have millions of digits in the denominator and numerator. Still, with enough computing power, proceeding with exact calculations would give arbitrarily close upper bounds on $\alpha_c$. 

The improvement to the lower bound uses a similar idea as in \cite[Theorem 3.5]{tree}. The authors show that if $X=\infty$ then there is an infinite sequence $(H_n)$ of connected subgraphs containing the root with
	\begin{enumerate}[label = (\alph*)]
		\item $|H_n| = n$ and
		\item there are at least $\lceil n/2\rceil $ vertices in $v\in H_n$ with $\eta_v =2$.
	\end{enumerate}
	The number of such subgraphs of size $n$ is counted by the $n$th Catalan number which is bounded by $4^n$. 
	
	The number of vertices with $\eta_v =2$ in a subgraph of size $n$ has the binomial distribution with parameters $n$ and $p = \alpha/2$. Thus, the probability of such a subgraph containing $k> n/2$ vertices with cars initially arriving to them is 
	$$\sum_{k>n/2} \binom{n}k p^k(1-p)^{n-k} \leq 2^n \sum_{k>n/2} p^{k} (1-p)^{n-k}.$$
	Above we use the fact that $\binom{n}{ k} \leq \binom{n}{\lceil n/2\rceil}$ for all $k \geq \lceil n/2 \rceil$. If $p< 1/2$, then $$\sum_{k > n/2} p^k (1-p)^{n-k}=  p^{\lceil n/2\rceil }(1-p)^{n-\lceil n/2\rceil } \sum_{k=0}^{n/2} \left(\f{p}{1-p}\right)^k  \leq  2C p^{n/2} (1-p)^{n/2}$$
	for $C = \sum_{k=0}^\infty (p / (1-p))^k = (1-p)/(1-2p)$. The `2' coefficient is to correct for the periodicity coming from the $(1-p)^{n - \lceil n/2\rceil}$ term. 
	
	Applying a union bound tells us that the probability that the sequence $(H_n)$ exists is bounded by
	$$\sum_{n=1}^\infty 4^n 2^n (2C p^{n/2} (1-p)^{n/2} )= 2C \sum_{n=1}^\infty (64 p (1-p) )^{n/2}.$$
	For $p < \f 12 - \f{\sqrt{15}}{8} \approx  0.015877$ the term $64p(1-p) < 1$ and the series is summable. It follows from the Borel-Cantelli lemma that there is no infinite sequence $(H_n)$ with the required properties, and thus $P(X=\infty) = 0$. Switching back to $\alpha = 2p$, this gives $\alpha_c > 2 (.015877) = .031754$. 

\end{proof}

\begin{remark}
The improvement to the lower bound in \cite[Theorem 3.5]{tree} is the $(1-p)^{n/2}$ term. The authors did not optimize to include it and instead of finding $p$ satisfying $64p(1-p) <1$, they required that $64p <1$. This gives their lower bound of $1/64$. 
\end{remark}

\subsection{Asymptotic behavior of $p_c(d)$}

\begin{proof}[Proof of \thref{thm:pcd}]
Both bounds are straightforward generalizations of \cite[Theorem 3.5]{tree}. For the upper bound, the idea is to compare to percolation that considers the $d^2$ vertices at distance two from the root. If one of these vertices has $\eta_v =2$, then it and its ancestor will be parked in. When $\alpha /2 > d^{-2}$, basic percolation theory tells us that there is almost surely an infinite connected path of occupied parking spots. As observed by Goldschmidt and Przykucki, the odd generations of the tree along this path almost surely have infinitely many cars arrive to them, of which infinitely many will reach the root since the spots on the path are parked in and the start of the path is some finite distance from $\rho$. It follows that $X=\infty$ almost surely.  Thus, $\alpha_c(d) \leq 2d^{-2}$. 

The lower bound follows the argument in the proof of \thref{thm:bounds} concerning the existence of a sequence of subgraphs $(H_n)$. The only modification needed is that on $\mathbb T_d$ the number of connected subgraphs containing the root with $n$ vertices is equal to the generalized Catalan number (see \cite{hilton1991catalan}) which is bounded by $\binom{dn}n$. Using a standard upper bound on binomial coefficients, we have 
$$\binom{dn}{n} \leq \left(\f{dne}{n}\right)^n = (ed)^n.$$
Thus, we can replace $4^n$ with $(ed)^n$ when applying a union bound. The rest of the quantities are unchanged. So, again letting $p = \alpha/2$, we have $X<\infty$ almost surely whenever 
$$(de)^n 2^n p^{n/2}(1-p)^{n/2} < 1.$$
Bounding the $(1-p)^{n/2}$ term by $1$, solving for $p$, then making the replacement $\alpha/2=p$ gives the claimed lower bound.  
\end{proof}

We conclude this section by showing that the value of $\alpha_c$ depends on how concentrated the arrival distribution is.

\begin{proof}[Proof of \thref{prop:different}]
	As in the proof of the upper bound in \thref{thm:pcd} we compare to percolation. Now that $\eta(\alpha) = 3$ with probability $\alpha/3$ we can consider the vertices at distance $3$ from the root. Whenever $\alpha /3 \geq d^{-3}$ there is almost surely an infinite connected path of spots that are parked in, which, by similar reasoning as before, implies that $X= \infty$ almost surely. Thus, $\alpha'_c \leq 3 d^{-3}$. 
\end{proof}

\section{The increasing convex order} \label{sec:icx}

\begin{proof}[Proof of \thref{thm:icx}]
Let $X_n$ and $X_n'$ be the number of arrivals to $\rho$ up to time $n$ for parking with arrival distributions $\eta$ and $\eta'$ as in \thref{prop:Xn}. We claim that it suffices to prove that 
	\begin{align}X_n \icx X_n' \text{  for all $n \geq 0$. } \label{eq:XicxX'}
	\end{align}
Suppose we show \eqref{eq:XicxX'}. It follows from the closure under sequences property \cite[Theorem 4.A.8.(c)]{SS} that, whenever $EX$ and $EX'$ are finite, we have $X \icx X'$. On the other hand, if $EX$ is infinite, then \eqref{eq:XicxX'} implies that $EX_n \leq EX_n'$ and thus $EX'$ is also infinite. By \cite[Theorem 3.4]{tree}, this implies that $P(X_n = \infty) =1= P(X_n'=\infty)$. So trivially we have $X\icx X'$.

To establish \eqref{eq:XicxX'}, we proceed inductively. By hypothesis we have
$$X_0 = \eta_\rho \icx \eta'_\rho = X_0'.$$
Now, supposing that $X_n \icx X_n'$. Since $\phi(x) = (x-1)^+$ is an increasing convex function on $[0,\infty)$, it follows from our inductive hypothesis and \cite[Theorem 4.A.8.(a)]{SS} that 
\begin{align}
(X_n-1)^+ \icx (X_n'-1)^+. \label{eq:+}	
\end{align}
Applying \eqref{eq:+} along with \cite[Theorem 4.A.9.]{SS} for random sums of i.i.d.\ random variables whose respective summands are dominated in the increasing convex order gives
\begin{align}\sum_{i=1}^Z (X_n^{(i)} -1)^+ \icx \sum_{i=1}^Z ( (X_n^{(i)})' - 1)^+\label{eq:icx_sum}.
\end{align}
Also, since $\eta_\rho \preceq \eta_\root'$, \cite[Theorem 4.A.8.(d)]{SS} and \eqref{eq:icx_sum} imply that
\begin{align}
\eta_\root + \sum_{i=1}^Z (X_n^{(i)} -1)^+ \icx \eta_\root'+ \sum_{i=1}^Z ( (X_n^{(i)})' - 1)^+\label{eq:icx_sum2}.
\end{align}
The left and right formulas in \eqref{eq:icx_sum2} are exactly the recursive equations for $X_{n+1}$ and $X_{n+1}'$ as in \eqref{eq:X_rec}. This gives \eqref{eq:XicxX'}, which concludes the argument.
\end{proof}

\begin{proof}[Proof of \thref{cor:max}]
Let $\eta := \eta(\alpha)$ be the car arrival distribution for Bernoulli parking. By \thref{thm:icx}, it suffices to prove for fixed $\alpha$ that $\eta\icx \eta'(\alpha):= \eta'$. This follows from a straightforward adaptation of \cite[Proposition 15 (b)]{JJ}.

 Let $\phi$ be an increasing convex function on $[0,\infty)$ with $\phi(0) = 0$. The last assumption is without loss of generality. Indeed, if we prove that $E\phi(\eta) \leq E \phi(\eta')$, then for an arbitrary increasing convex function $\psi$ we define $\bar \psi(x) = \psi(x) - \psi(0)$. It is easy to see that $E \bar \psi(\eta') \leq E\bar  \psi(\eta)$ if and only if $E \psi(\eta') \leq E \psi(\eta)$. 
   
First, we have
  \begin{align}
    E\phi(\eta) = (\alpha/2) \phi(2). \label{eq:Eeta}
  \end{align}
As for $\eta'$, let $a=E[ \eta' \mid \eta' \geq 2]$. Since $\phi(0) = 0$ and $P(\eta' =1) = 0$, we can condition and apply Jensen's inequality
 \begin{align}
 E \phi(\eta') &= E[\phi(\eta') \mid \eta' \geq 2]P(\eta'\geq 2) \\
 			  & \geq \phi( a) P(\eta'\geq 2). \label{eq:jensen}
 \end{align}
As $a \geq 2$ and $\phi$ is convex, the point $(a,\psi(a))$ lies above the secant line connecting $(0,0)$ and $(2,\phi(2))$. It follows that $a \phi(2)/2 \leq \phi(a)$. Applying this to \eqref{eq:jensen} gives  
$$E \phi(\eta') \geq a \f{\phi(2)}2 P(\eta' \geq 1).$$
Notice that $a P(\eta' \geq 1) = E \eta' = \alpha$, so we have $E \phi(\eta') \geq (\alpha/2) \phi(2) = E \phi(\eta)$ by \eqref{eq:Eeta}. Thus, $\eta \icx \eta'$. 
%
\end{proof}

\subsection*{Acknowledgements} 
 Thanks to Hanbaek Lyu and Olivier H\'enard for helpful feedback.

%
%
%
%
%
%
%

\bibliographystyle{amsalpha}
\bibliography{tree_parking}

\newcommand{\etalchar}[1]{$^{#1}$}
\providecommand{\bysame}{\leavevmode\hbox to3em{\hrulefill}\thinspace}
\providecommand{\MR}{\relax\ifhmode\unskip\space\fi MR }
\providecommand{\MRhref}[2]{%
  \href{http://www.ams.org/mathscinet-getitem?mr=#1}{#2}
}
\providecommand{\href}[2]{#2}
\begin{thebibliography}{CEGM83}

\bibitem[BL91]{BL4}
Maury Bramson and Joel~L. Lebowitz, \emph{Asymptotic behavior of densities for
  two-particle annihilating random walks}, J. Statist. Phys. \textbf{62}
  (1991), no.~1-2, 297--372. \MR{1105266}

\bibitem[CEGM83]{collet1983study}
P~Collet, J-P Eckmann, V~Glaser, and A~Martin, \emph{Study of the iterations of
  a mapping associated to a spin glass model}, Les rencontres
  physiciens-math{\'e}maticiens de Strasbourg-RCP25 \textbf{33} (1983),
  117--142.

\bibitem[CG19]{chen2019parking}
Qizhao Chen and Christina Goldschmidt, \emph{Parking on a random rooted plane
  tree}, 2019.

\bibitem[CH19]{curien2019phase}
Nicolas Curien and Olivier Hénard, \emph{The phase transition for parking on
  {G}alton--{W}atson trees}, 2019.

\bibitem[CRS14]{Cabezas2014}
M.~Cabezas, L.~T. Rolla, and V.~Sidoravicius, \emph{Non-equilibrium phase
  transitions: Activated random walks at criticality}, Journal of Statistical
  Physics \textbf{155} (2014), no.~6, 1112--1125.

\bibitem[DGJ{\etalchar{+}}19]{parking}
Michael Damron, Janko Gravner, Matthew Junge, Hanbaek Lyu, and David Sivakoff,
  \emph{Parking on transitive unimodular graphs}, The Annals of Applied
  Probability \textbf{29} (2019), no.~4, 2089--2113.

\bibitem[DH17]{persi_parking}
Persi Diaconis and Angela Hicks, \emph{Probabilizing parking functions},
  Advances in Applied Mathematics \textbf{89} (2017), 125 -- 155.

\bibitem[GP19]{tree}
Christina Goldschmidt and Micha{\l} Przykucki, \emph{Parking on a random tree},
  Combinatorics, Probability and Computing \textbf{28} (2019), no.~1, 23--45.

\bibitem[HMP19]{Hu_2019}
Yueyun Hu, Bastien Mallein, and Michel Pain, \emph{An exactly solvable
  continuous-time {D}errida--{R}etaux model}, Communications in Mathematical
  Physics (2019).

\bibitem[HP91]{hilton1991catalan}
Peter Hilton and Jean Pedersen, \emph{Catalan numbers, their generalization,
  and their uses}, The mathematical intelligencer \textbf{13} (1991), no.~2,
  64--75.

\bibitem[HS18]{hu2018free}
Yueyun Hu and Zhan Shi, \emph{The free energy in the derrida--retaux recursive
  model}, Journal of Statistical Physics \textbf{172} (2018), no.~3, 718--741.

\bibitem[JJ18]{JJ}
Tobias Johnson and Matthew Junge, \emph{Stochastic orders and the frog model},
  Ann. Inst. Henri Poincar\'e Probab. Stat. \textbf{54} (2018), no.~2,
  1013--1030. \MR{3795075}

\bibitem[Jon19]{jones2019runoff}
Owen~Dafydd Jones, \emph{Runoff on rooted trees}, Journal of Applied
  Probability \textbf{56} (2019), no.~4, 1065--1085.

\bibitem[KW66]{KW}
Alan~G. Konheim and Benjamin Weiss, \emph{An occupancy discipline and
  applications}, SIAM Journal on Applied Mathematics \textbf{14} (1966), no.~6,
  1266--1274.

\bibitem[LC95]{AB1}
Benjamin~P. Lee and John Cardy, \emph{Renormalization group study of the
  ${A+B}\to0$ diffusion-limited reaction}, Journal of Statistical Physics
  \textbf{80} (1995), no.~5, 971--1007.

\bibitem[LP16]{uniform_parking}
Marie-Louise Lackner and Alois Panholzer, \emph{Parking functions for
  mappings}, Journal of Combinatorial Theory, Series A \textbf{142} (2016),
  1--28.

\bibitem[Mar02]{Marchand}
R.~Marchand, \emph{Strict inequalities for the time constant in first passage
  percolation}, Ann. Appl. Probab. \textbf{12} (2002), no.~3, 1001--1038.
  \MR{1925450}

\bibitem[OZ78]{chem1}
A.~A. Ovchinnikov and Ya.~B. Zeldovich, \emph{Role of density fluctuations in
  bimolecular reaction kinetics}, Chem. Phys. \textbf{28} (1978), no.~1--2,
  215--218.

\bibitem[PRS19]{parking_on_integers}
Micha{\l} Przykucki, Alexander Roberts, and Alex Scott, \emph{Parking on the
  integers}, arXiv preprint arXiv:1907.09437 (2019).

\bibitem[SP02]{stanley3}
Richard~P Stanley and Jim Pitman, \emph{A polytope related to empirical
  distributions, plane trees, parking functions, and the associahedron},
  Discrete \& Computational Geometry \textbf{27} (2002), no.~4, 603--602.

\bibitem[SS07]{SS}
Moshe Shaked and J.~George Shanthikumar, \emph{Stochastic orders}, Springer
  Series in Statistics, Springer, New York, 2007. \MR{2265633 (2008g:60005)}

\bibitem[Sta97]{stanley1}
Richard~P Stanley, \emph{Parking functions and noncrossing partitions},
  Electron. J. Combin \textbf{4} (1997), no.~2, R20.

\bibitem[Sta98]{stanley2}
\bysame, \emph{Hyperplane arrangements, parking functions and tree inversions},
  Mathematical essays in honor of Gian-Carlo Rota, Springer, 1998,
  pp.~359--375.

\bibitem[vdBK93]{BK}
J.~van~den Berg and H.~Kesten, \emph{Inequalities for the time constant in
  first-passage percolation}, Ann. Appl. Probab. \textbf{3} (1993), no.~1,
  56--80. \MR{1202515}

\end{thebibliography}

\end{document}